\pdfoutput=1

\documentclass[letterpaper, 12pt]{article}
\usepackage{amssymb,amsmath}
\usepackage{epsf}
\usepackage{times,bm,mathrsfs}  
\usepackage{amsmath}
\usepackage{amssymb}
\usepackage{amsfonts}
\usepackage{mathrsfs}
\usepackage{bbm}
\usepackage{color}
\usepackage{graphicx}
\usepackage{amscd}
\usepackage{epsfig}

\definecolor{Red}{rgb}{1,0,0}
\definecolor{Blue}{rgb}{0,0,1}
\definecolor{Olive}{rgb}{0.41,0.55,0.13}
\definecolor{Green}{rgb}{0,1,0}
\definecolor{MGreen}{rgb}{0,0.8,0}
\definecolor{DGreen}{rgb}{0,0.55,0}
\definecolor{Yellow}{rgb}{1,1,0}
\definecolor{Cyan}{rgb}{0,1,1}
\definecolor{Magenta}{rgb}{1,0,1}
\definecolor{Orange}{rgb}{1,.5,0}
\definecolor{Violet}{rgb}{.5,0,.5}
\definecolor{Purple}{rgb}{.75,0,.25}
\definecolor{Brown}{rgb}{.75,.5,.25}
\definecolor{Grey}{rgb}{.5,.5,.5}
\definecolor{Black}{rgb}{0,0,0}

\newcommand{\bcal}{\mathcal{B}}

\newcommand{\dcal}{\mathcal{D}}

\newcommand{\rcal}{\mathcal{R}}

\newcommand{\tcal}{\mathcal{T}}
\newcommand{\ucal}{\mathcal{U}}
\newcommand{\vcal}{\mathcal{V}}

\newcommand{\ycal}{\mathcal{Y}}

\newcommand{\real}{\mathbb{R}}

\newcommand{\eps}{\varepsilon}

\newcommand{\ind}{\mathbbm{1}}

\newcommand{\bdm}{\begin{displaymath}}
\newcommand{\edm}{\end{displaymath}}
\newcommand{\bea}{\begin{eqnarray*}}
\newcommand{\eea}{\end{eqnarray*}}
\newcommand{\bean}{\begin{eqnarray}}
\newcommand{\eean}{\end{eqnarray}}

\newcommand{\bfx}{\mathbf{x}}
\newcommand{\bfy}{\mathbf{y}}

\newcommand{\prob}{\mathbb{P}}
\newcommand{\expec}{\mathbb{E}}
\newcommand{\E}{\mathbb{E}}
\renewcommand{\P}{\mathbb{P}}
\newcommand{\var}{\mathrm{Var}}

\newtheorem{theorem}{Theorem}
\newtheorem{proposition}{Proposition}

\newtheorem{definition}{Definition}
\newtheorem{example}{Example}
\newtheorem{lemma}{Lemma}

\newtheorem{remark}{Remark}

\newtheorem{assumption}{Assumption}

\newenvironment{proof}{\noindent{\textbf{Proof:}}}{$\blacksquare$\vskip\belowdisplayskip}
\newcommand{\itemname}[1]{$\mathrm{[#1]}$}

\newcommand{\weight}{\mu}
\newcommand{\eweight}{\hat \weight}
\newcommand{\distd}{\hat d}
\newcommand{\dist}{d}
\newcommand{\underm}{\underline{M}}
\newcommand{\overm}{\overline{M}}

\newcommand{\nsqneq}{[n]^2_{\neq}}
\newcommand{\nfourneq}{[n]^4_{\neq}}

\newcommand{\path}{\mathrm{Path}}
\newcommand{\mgf}{\Phi}

\newcommand{\maxd}{M}

\newcommand{\assump}{\mathrm{A}}
\newcommand{\gpmassump}{\mathrm{GPM}}
\newcommand{\maxlambda}{\overline{\lambda}}
\newcommand{\minlambda}{\underline{\lambda}}
\newcommand{\maxuu}{\overline{U}}
\newcommand{\minuu}{\underline{U}}
\newcommand{\bin}{\widehat{B}}
\newcommand{\short}{\text{\L}^-_2}
\newcommand{\notshort}{\text{\L}^+_2}

\addtocounter{page}{-1}

\begin{document}

\title{\vspace{0cm}
Identifiability and inference 
of non-parametric
rates-across-sites models 
on large-scale phylogenies
}

\author{
Elchanan Mossel\footnote{
U.C. Berkeley and Weizmann Institute of Science. Supported by DMS
0548249 (CAREER) award, by DOD ONR grant N000141110140, by ISF
grant 1300/08 and by ERC PIRG04-GA-2008-239137 grant.}
\and
Sebastien Roch\footnote{Department of Mathematics
and Bioinformatics Program, UCLA.
Work supported by NSF grant DMS-1007144.}
}
\maketitle

\begin{abstract}
Mutation rate variation across loci 
is well known to cause difficulties, notably
identifiability issues, 
in the reconstruction of evolutionary trees
from molecular sequences. 
Here we introduce a new approach for 
estimating general rates-across-sites 
models.
Our results imply, in particular, that large phylogenies
are typically identifiable under rate variation.
We also derive sequence-length requirements for
high-probability reconstruction.

Our main contribution is a novel algorithm that clusters 
sites according to their mutation rate. Following this site clustering step, standard reconstruction techniques can be used to recover the phylogeny.
Our results rely on a basic insight: 
that, for large trees, certain site statistics
experience concen\-tra\-tion-of-measure phenomena.
\end{abstract}

\thispagestyle{empty}

\clearpage

\section{Introduction}\label{section:introduction}

The evolutionary history of living organisms is
typically represented graphically by a {\em phylogeny},
a tree whose branchings indicate past speciation events.
The inference of phylogenies 
based on molecular sequences 
extracted from extant species is a major task of computational biology. 
Among the many biological phenomena 
that complicate this
task, one that has received much attention
in the statistical phylogenetics literature
is the variation in mutation rate
across sites in a genome.
(See related work below.)
Such variation is generally attributed
to unequal degrees of selective pressure.
As we describe formally below, mathematically
this phenomenon can be modeled as a {\em mixture} 
of phylogenies. That is, interpreting branch length as a measure of the amount of evolutionary change,
rates-across-sites (RAS) models posit that
all sites in a genome evolve according to a common 
tree topology, but branch lengths for a given site are scaled
by a random factor.

Here we introduce a new approach
for estimating RAS models.
Our main contribution is a novel algorithm
which {\em clusters} the sites according to their 
mutation rate.
We show that our technique may be used to reconstruct phylogenies. Indeed, following the site clustering step, standard reconstruction techniques
can be employed to recover a phylogeny on the 
unmixed subset of sites obtained.
Our results rely on the following basic insight: there exist simple site-wise statistics
that experience con\-cen\-tra\-tion-of-measure phenomena.
Consequently, our techniques only hold {\em in the large-tree limit}.

Concentration has been used extensively in statistical phylogenetics.
However its typical use is in the
\emph{large-sample limit}, that is, as the sequence length grows to infinity, for instance
in order to show that so-called evolutionary distance estimates are accurate given sufficiently
long sequences
(see e.g.~\cite{ErStSzWa:99a}).
Instead, we consider here concentration
in what we call the \emph{large-tree limit},
that is, as the number of leaves goes to infinity.
Note that the latter is trickier to analyze.
Indeed, whereas different sites are usually assumed 
to evolve independently,
leaf states are {\em not} independent.
To the best of our knowledge, this is the first use of this type of concentration in the context
of phylogenetics.

Our results imply, in particular, that large phylogenies
are typically identifiable under rate variation.
We also derive sequence-length requirements for
high-probability reconstruction.

\subsection{Related work}\label{section:related}

Most prior theoretical work on mixture models 
has focused on the question of {\em identifiability}. 
A class of phylogenetic models is identifiable 
if any two models in the class produce different 
data distributions.
It is well-known that unmixed phylogenetic models 
are typically identifiable~\cite{Chang:96}.
{\em This is not the case in general for mixtures of phylogenies.}
For instance, Steel et al.~\cite{StSzHe:94} showed that for any two trees
one can find a random scaling on each of them such that their data distributions
are identical. Hence it is hopeless in general to reconstruct
phylogenies under mixture models.
See also~\cite{EvansWarnow:04,MatsenSteel:07,MaMoSt:08,StefankovicVigoda:07a,StefankovicVigoda:07b,Steel:09} for further examples of this type. 

However the negative examples constructed in the references above are not necessarily typical. 
They use special features of the mutation models (and their invariants) 
and allow themselves quite a bit of flexibility in setting up the topologies and branch lengths.
In fact, recently a variety of more standard mixture models have been shown to be
identifiable. These include the common GTR+Gamma model~\cite{AlAnRh:08,WuSusko:10} and
GTR+Gamma+I model~\cite{ChaiHousworth:11},
as well as some covarion models~\cite{AllmanRhodes:06},
some group-based models~\cite{AlPeRhSu:11},
and so-called $r$-component identical tree mixtures~\cite{RhodesSullivant:10}.
Although these results do not provide practical algorithms for reconstructing
the corresponding mixtures, they do give hope that these problems
may be tackled successfully.

Beyond the identifiability question, there seems to have been little rigorous
work on reconstructing phylogenetic mixture models. One positive
result is the case of the molecular clock assumption with across-sites
rate variation~\cite{StSzHe:94}, although no sequence-length requirements are provided.
There is a large body of work on practical reconstruction
algorithms for various types of mixtures, notably rates-across-sites models
and covarion-type models, using mostly likelihood and bayesian methods. See e.g.~\cite{Felsenstein:04} for references.
But the optimization problems they attempt to solve are 
likely NP-hard~\cite{ChorTuller:06,Roch:06}.
There also exist many techniques for testing for the presence of a mixture
(for example, for testing for rate heterogeneity), but such tests typically require
the knowledge of the phylogeny. See e.g.~\cite{HuelsenbeckRannala:97}.

Here we give both identifiability and reconstruction
results. Whereas Steel et al.~\cite{StSzHe:94}
show that any two fixed trees can be made indistinguishable
with an appropriate (arbitrarily complex) choice
of scaling distributions, we show in essence
that, given a fixed rate distribution
(or a well-behaved class of rate distributions),
sufficiently large trees are typically distinguishable. 
After a draft of our results 
were circulated~\cite{MosselRoch:08}, 
related results for large trees 
were established by 
Rhodes and Sullivant~\cite{RhodesSullivant:10}
using different techniques. 
In particular, our technical assumptions are
similar in spirit to the genericity condition
in~\cite{RhodesSullivant:10}. Although our
genericity assumptions are stronger, they allow 
an efficient reconstruction of the model
and explicit bounds on sequence-length requirements.
Note moreover that our results apply to
general, possibly continuous, nonparametric
rate distributions. 

The proof of our main results relies
on the construction of a {\em site clustering statistic} that 
discriminates between different rates.
A similar statistic was also used 
in~\cite{SteelSzekely:06} in a different context. 
However, 
in contrast to~\cite{SteelSzekely:06}, our main 
reconstruction result
requires that
a site clustering statistic be constructed
based only on data generated by
the mixture---that is, {\em without} prior
knowledge of the model.

\subsection{Overview of techniques}\label{section:overview}

\paragraph{A simplified setting} 
To illustrate our main ideas, we first consider a simple two-speed model. 
Assume that molecular sequences have two types of sites:
``slow'' and ``fast.''
Both types of sites evolve independently
by single substitution on a common evolutionary tree according, say,
to a standard Jukes-Cantor model of substitution,
but the fast ones
evolve three times as fast. 
See Section~\ref{section:preliminaries} for
a formal definition of the Jukes-Cantor model.
To keep things simple, assume for now that
the evolutionary tree is a complete binary tree with $n = 2^h$ leaves,
where $h$ is the number of levels. 
(Note that our results apply to much more general
rate distributions.
We also discuss how to deal with general trees.
See below.)

Our approach is based on the following question:
Is it possible to tell {\em with high confidence} which sites are slow or fast, with no prior
knowledge of the phylogeny that generated them?
Perhaps surprisingly, the answer is {\em yes}---at least 
for large trees.
This far-reaching observation does not seem to have been made previously.

To see how this works, assume for the time being that we know the phylogeny.
We will show how to remove this assumption below.
Take a pair of leaves $a,b$. The effect of the speed of a
site can be seen in the probability of agreement between $a$ and $b$:
the leaves agree more often on slow-evolving sites.
Hence, if a site shows agreement between $a$ and $b$, one may deduce that
the site is more likely to be slow-evolving. But this is too little information
to infer {\em with high confidence} the speed of a site. Instead, one may
look at a larger collection of pairs of leaves
and consider the statistic that counts how many of them agree on a given site.
The idea is that a large number of agreements should indicate a slow site.
For this scheme to work accurately, we require two properties from this statistic:
{\em separation} and {\em concentration}.  By separation, we mean that the expected value
of the statistic should be different on slow and fast sites---in order to distinguish them.
By concentration, we mean that the statistic should lie very close to its expectation. These two properties
produce a good site clustering test.
To satisfy them, the pairs of leaves involved must be chosen carefully.

\paragraph{Separation and concentration} To obtain separation, it is natural to use only pairs of ``close'' leaves.
Indeed, leaves that are far away are practically independent and the speed
of a site has very little noticeable effect on their agreement.
As for concentration, what one needs 
is the kind of conditions that give rise
to the central limit theorem: a large sum of small independent contributions.
For symmetric models such as the Jukes-Cantor model,
the {\em agreement events} on two pairs of leaves $(a,b)$ 
and $(c,d)$ are independent
as long as the paths between $(a,b)$ and $(c,d)$ do not intersect.
Therefore, we are led to consider the following statistic:
count how many cherries (that is, sister leaves) agree and
divide by the
total number of cherries to obtain a fraction.
One can show from the considerations above that such a statistic is highly concentrated.


\paragraph{Unknown, general tree} 
However our derivation {\em so far} has relied heavily on
two unsatisfied premises: 
\begin{enumerate}
\item That the tree is known. This is of course not 
the case since our ultimate goal is precisely
to reconstruct the phylogeny.

\item And that the tree is complete. In particular,
our argument uses the fact that complete binary trees 
contain many cherries. But general trees
may have very few cherries.
\end{enumerate}
{\em Perhaps surprisingly, neither of these conditions is necessary.} 
The bulk of the technical contributions of this paper 
lie in getting rid of these assumptions.
We show in particular
how to construct a site clustering statistic similar to the one above directly
from the data without prior knowledge of the tree. 
At a high
level, all one needs is to select a large collection of ``sufficiently correlated'' pairs of leaves
and then ``dilute'' them to discard pairs that are too close to each other.
This leads to a highly concentrated site-wise statistic.
See Section~\ref{section:preliminaries} for a statement of our results.

\section{Definitions and Results}\label{section:preliminaries}

\subsection{Basic Definitions}

\paragraph{Phylogenies}
A phylogeny is a graphical representation 
of the speciation history of a group of organisms.
The leaves typically correspond to current species. 
Each branching indicates a speciation event. 
Moreover we associate to each edge a 
positive weight.
This weight can be thought roughly as the
time elapsed on the edge multiplied by the
mutation rate which may also depend on the edge. 
More formally:
\begin{definition}[Phylogeny]\label{def:phylo}
A \emph{phylogeny} $T = (V,E;L,\weight)$ is a tree with vertex set $V$, edge set $E$ and
$n$ (labelled) leaves $L = [n] = \{1,\ldots,n\}$ such that 1) the degree of all internal vertices
$V-L$ is exactly $3$, and 2) the edges are assigned weights $\weight : E \to (0,+\infty)$.
We let $\tcal[T] = (V,E;L)$ be the
{\em topology} of $T$.
A phylogeny is naturally equipped with a so-called 
{\em tree metric} on the leaves
$\dist : L\times L \to (0,+\infty)$ defined as follows
\begin{equation*}
\forall u,v \in L,\ \dist(u,v) = \sum_{e\in\path_T(u,v)} \weight_e,
\end{equation*}
where $\path_T(u,v)$ is the set of edges on the path between
$u$ and $v$ in $T$.
We will refer to $\dist(u,v)$ as the {\em evolutionary
distance} between $u$ and $v$.
Since under the assumptions above there is a one-to-one
correspondence between $\dist$ and $\weight$
(see e.g.~\cite{SempleSteel:03}),
we write either $T = (V,E;L,\dist)$ or $T = (V,E;L,\weight)$.
We also sometimes use the natural extension of $\dist$ to the internal vertices
of $T$. 
\end{definition}
We will sometimes restrict ourselves to the following standard special case.
\begin{definition}[Regular Phylogenies]
Let $0 < f \leq g < +\infty$.
We denote by $\tcal_{f,g}$ the set of phylogenies $T = (V,E; L,\weight)$ such that
$\forall e\in E$, $f \leq \weight_e \leq g$.
\end{definition}

\paragraph{Poisson Model}
A standard model of DNA sequence evolution
is the following \emph{Poisson model}. See e.g.~\cite{SempleSteel:03}.
\begin{definition}[Poisson Model]\label{def:mmt}
Consider the following stochastic
process.
We are given a phylogeny $T = (V,E; [n],\weight)$
and a finite set $\rcal$ with $r$ elements.
Let $\pi$ be a probability distribution
on $\rcal$.
Let $Q \in \real^{r\times r}$ be the following \emph{rate matrix}
\begin{displaymath}
Q_{x y} =
\left\{
\begin{array}{ll}
\pi_y, & \text{if $x \neq y$},\\
\pi_y - 1, & \text{o.w.}
\end{array}
\right.
\end{displaymath}
Associate to each edge $e\in E$ the stochastic matrix
\begin{displaymath}
\left[M(e)\right]_{x y} = \left[\exp\left(\weight_e Q\right)\right]_{x y} =
\left\{
\begin{array}{ll}
\pi_x + (1 - \pi_x)e^{-\weight_e}, & \text{if $x = y$},\\
\pi_y(1 - e^{- \weight_e}) , & \text{o.w.}
\end{array}
\right.
\end{displaymath}
The process runs as follows.
Choose an arbitrary root $\rho \in V$.
Denote by $E_\downarrow$ the set $E$ directed away from the root.
Pick a state for the root according
to $\pi$. Moving away from the root toward the leaves,
apply the channel $M(e)$ to each edge $e$ independently.
Denote the state so obtained $\sigma_V = (\sigma_v)_{v\in V}$. In particular,
$\sigma_{[n]}$ is the state at the leaves.
More precisely, the joint distribution of $\sigma_V$ is given by
\begin{equation*}
\mu_V(\sigma_V) = \pi_{\rho}(\sigma_\rho)
\prod_{e = (u,v) \in E_\downarrow} \left[M(e)\right]_{\sigma_{u} \sigma_{v}}.
\end{equation*}
For $W \subseteq V$, we denote by $\mu_W$ the marginal
of $\mu_V$ at $W$. Under this model, the weight $\weight_e$ is
the expected number of substitutions on edge $e$ 
in a related
continuous-time process. 
The \emph{$r$-state Poisson model} is the
special case when $\pi$ is the uniform
distribution over $\rcal$.
In that case, 
we denote the distribution of $\sigma_V$ by
$\dcal[T,r]$. When $r$ is clear from the context,
we write instead $\sigma_V \sim \dcal[T]$.
\end{definition}
More generally, we take
$k$ independent samples $(\sigma^{i}_{V})_{i=1}^k$
from the model above, that is, $\sigma^1_V, \ldots, \sigma^k_V$
are i.i.d.~$\dcal[T,r]$.
We think of $(\sigma_v^i)_{i=1}^k$
as the sequence at node $v \in V$. Typically,
$\rcal = \{\mathrm{A}, \mathrm{G},\mathrm{C},\mathrm{T}\}$ and
the model describes how DNA sequences stochastically evolve by point mutations
along an evolutionary tree---under the assumption 
that each site in the sequences evolves
independently.
\begin{example}[CFN and Jukes-Cantor Models]
The special case $r=2$ corresponds to the so-called CFN model.
The special case $r=4$ is the well-known Jukes-Cantor model.
\end{example}
We fix $r$ throughout.
\begin{remark}
We discuss the more general GTR model
in the concluding remarks.
\end{remark}

\paragraph{Rates-across-sites model}
We introduce the basic rates-across-sites model 
which will be the focus of this paper.
We will use the following definition.
\begin{definition}[Phylogenetic Scaling]
Let $T = (V,E;[n],\weight)$ be a phylogeny and
$\Lambda$, a constant in $[0,+\infty)$. Then we denote by
$\Lambda T$ the phylogeny obtained by scaling the weights of $T$
by $\Lambda$, that is, $\Lambda T = (V,E;[n],\Lambda \weight)$.
\end{definition}
\begin{definition}[Rates-Across-Sites Model (see e.g.~\cite{StSzHe:94})]
In the \emph{generalized Poisson model}
we are given
a phylogeny $T$
and a \emph{scaling factor} $\Lambda$, that is,
a random variable on
$[0,+\infty)$.
Let $\Lambda_1, \ldots, \Lambda_k$ be i.i.d.~copies of $\Lambda$.
Conditioned on $\Lambda_1, \ldots, \Lambda_k$, 
the samples $(\sigma^{i}_{V})_{i=1}^k$ generated under this model
are independent with
$\sigma^j_V \sim \dcal[\Lambda_j T]$, $j=1,\ldots,k$.
We denote by $\dcal[T, \Lambda, r]$ the probability
distribution
of $\sigma^1_V$. We also let 
$\overline{\dcal}[T, \Lambda, r]$ be the probability
distribution of $\sigma^1_L$.
\end{definition}

\subsection{Main results}

\paragraph{Tree identifiability}
To provide a {\em uniform} bound on the minimum
tree size required for our identifiability result to hold,
we make explicit assumptions on the mutation model.
For $s \geq 0$, let 
$$
\mgf(s) = \E\left[e^{-s \Lambda}\right],
$$
be the {\em moment generating function} (or
one-sided Laplace transform) of the scaling
factor $\Lambda$. The probability distribution
of $\Lambda$ is determined
by $\mgf$. See e.g.~\cite{Billingsley:95}.
We normalize $\Lambda$ so that
$$
- \mgf'(0) = \E\left[\Lambda\right]  = 1.
$$
In particular, $\Lambda$ is not identically
$0$ and $\mgf$ is continuous and strictly decreasing.
\begin{assumption}\label{assump:general}
Let
$0 < f \leq g  < +\infty$, 
and $M > 0$.
The following set of assumptions on
a generalized Poisson model
will be denoted by $\assump(
f,g,M)$: 
\begin{enumerate}
\item {\em Regular phylogeny:} 
The phylogeny $T = (V,E;[n],\weight) $ is in
$\tcal_{f,g}$.

\item {\em Mass close to $0$:}
We have that
$$
\mgf^{-1}\left(e^{-6g}\right) \leq \maxd. 
$$
\end{enumerate}
In words, an evolutionary distance
of $M$ under $\Lambda$-scaling
produces a correlation corresponding 
to an evolutionary
distance of at least $6g$ without the scaling.
We denote by $\gpmassump(f, g, M, n_0)$
the set of generalized Poisson models
satisfying $\assump(f, g, M)$ with at least $n_0$ leaves.
\end{assumption}  
\begin{remark}\label{rem:zero}
Note that the results in~\cite{StSzHe:94} indicate
that appropriate conditions are needed to
obtain a tree identifiability result in the generalized
Poisson model when the random scaling
is unknown. We do not claim that the conditions
above are minimal.
The first assumption is meant to ensure
that there is enough signal 
to reconstruct the tree. 
The second assumption bounds the distortion
of the random scaling for evolutionary distances
corresponding to short paths. It essentially implies
that the probability mass of $\Lambda$ close to $0$
is bounded. In particular, note that if the probability mass below
$$
\eps = -\frac{1}{M} \ln \left(\frac{e^{-6g} - \delta}{1-\delta}\right),
$$
is less than $\delta$ (for $\delta < e^{-6g}$),
then
$$
\mgf(M)
\leq \delta + (1 - \delta) e^{-\eps \maxd}
\leq e^{-6g},
$$
and the assumption is satisfied.
Conversely, if $\Lambda$ satisfies the second
assumption then the probability mass $\delta$ 
below $\eps$ (for $\eps < 6g/M$) must be such that
$$
\delta \leq e^{-(6g - \eps M)},
$$
since
$$
e^{-6g} \geq \mgf(M) \geq \delta e^{-\eps M}. 
$$
\end{remark}
\begin{remark}
The second assumption implicitly implies that
$$
\P[\Lambda = 0] \leq e^{-6g}.
$$
This is in fact not necessary.
By first removing all invariant sites, 
it should be possible to extend our main theorem 
to moment generating functions
of the form
$$
\mgf(s) = \alpha + (1 - \alpha)\mgf_+(s),
$$
where $0 \leq \alpha < 1$ is uniformly bounded
away from $1$ and $\mgf_+$ satisfies
the assumptions above.
Indeed, on a large phylogeny,
it is extremely unlikely to
produce an invariant site
using a positive scaling factor. 
Hence removing all invariant sites 
has the effect of essentially
restricting the dataset to the positive part
of the distribution of $\Lambda$. We leave the
details to the reader. Given this observation,
in the rest of the manuscript one can assume
that
$$
\P[\Lambda = 0] = 0.
$$
\end{remark}
\begin{theorem}[Tree identifiability]
\label{thm:1}
Fix
$0 < f \leq g  < +\infty$, 
and $M > 0$.
Then, there exists 
$n_0(f, g, M) \geq 1$ such that,
if $(T, \Lambda)$
and $(T',\Lambda')$ are in
$\gpmassump(f, g, M, n_0)$
with $\tcal[T] \neq \tcal[T']$, then
$$
\overline{\dcal}[T,\Lambda,r] \neq \overline{\dcal}[T',\Lambda',r].
$$
(Recall that $\overline{\dcal}[T,\Lambda,r]$ 
denotes the distribution {\em at the leaves}.)
\end{theorem}
\begin{remark}
Note that we allow $\Lambda \sim \Lambda'$
(where $\sim$ denotes equality in distribution).
This is the sense in which our result is 
a {\em tree} identifiability result.
\end{remark}
\begin{remark}
Note that our identifiability result applies
only to {\em sufficiently large} phylogenies. 
Computing $n_0$ from our techniques
is difficult (and in general depends
on the parameters $f,g,M$). One could estimate
the required size by running
the reconstruction algorithm below on
simulated data for various sizes and
parameters. We leave such empirical
studies for future work.
\end{remark}
The proof of our main theorem relies on the following reconstruction result.

\paragraph{Tree reconstruction} Moreover,
we give a stronger result implying that
the phylogeny can be reconstructed with high confidence
using polynomial length sequences in polynomial time.
The proof appears in Sections~\ref{section:existence},
\ref{section:constructing}
and \ref{section:full}.
\begin{theorem}[Tree Reconstruction]\label{thm:2}
Under Assumption~\ref{assump:general},
for all $0 < \delta < 1$,
there is a $\gamma_k>0$ large enough
so that the topology of the tree
can be reconstructed in polynomial
time using $k = n^{\gamma_k}$
samples, except with probability
$\delta$.
\end{theorem}
\begin{remark}
Once the tree has been estimated,
one can also infer the
rate distribution. Details are
left to the interested reader.
\end{remark}

\section{Site clustering statistic: 
Existence and properties}
\label{section:existence}

In this section, we introduce our main site clustering statistic and show
that it is concentrated.
Let $0 < f \leq g < +\infty$ and $M > 0$. 
In this section, we fix a phylogeny $T = (V,E;[n],\weight)$ in
$\tcal_{f,g}$.
We let $(\sigma^{i}_{L})_{i=1}^k$ 
be $k$ i.i.d.~samples from 
$\overline{\dcal}[T, \Lambda, r]$
where the generalized Poisson model $(T,\Lambda)$
satisfies Assumption~\ref{assump:general}.
Moreover, let $\Lambda_1,\ldots,\Lambda_k$
be the i.i.d.~scaling factors corresponding to the
$k$ samples above.

\paragraph{Some notation}
We will also use the notation
$[n]^2 = \{(a,b) \in [n]\times [n]\ :\ a \leq b\}$,
$[n]^2_{=} = \{(a,a)\}_{a\in [n]}$,
and $\nsqneq = [n]^2 - [n]^2_{=}$.
We also denote by $\nfourneq$ the set of pairs
$(a,b), (c,d)\in \nsqneq$ such that
$(a,b) \neq (c,d)$ (as pairs).
For $\alpha > 0$, we let
\begin{displaymath}
\Upsilon_\alpha 
= \{(a,b) \in \nsqneq\ :\ \dist(a,b) \leq \alpha\},
\end{displaymath}
be all pairs of leaves in $T$ at evolutionary 
distance at most $\alpha$.
Let $p_\infty = 1 - q_\infty$ where
$q_\infty = \sum_{x \in \rcal} \pi_x^2$.

\subsection{What makes a good site clustering statistic?}

For a site $i = 1,\ldots, k$, consider a statistic of the form
\begin{equation}\label{eq:ucal}
\ucal_i =
\frac{1}{|\Upsilon|} \sum_{(a,b) \in \Upsilon} 
p_\infty^{-1}[\ind\{\sigma^i_a = \sigma^i_b\} - q_\infty],
\end{equation}
where
$\Upsilon \subseteq \nsqneq$,
is a subset of pairs of distinct leaves
independent of $i$. Using the expression
for the transition matrix given in Definition~\ref{def:mmt},
note that
\begin{eqnarray*}
\E[\ucal_i\,|\,\Lambda_i]
&=& \frac{1}{|\Upsilon|} \sum_{(a,b) \in \Upsilon} 
p_\infty^{-1}\left[\E[\ind\{\sigma^i_a = \sigma^i_b\}\,|\,\Lambda_i]
 - q_\infty\right]\\
&=& \frac{1}{|\Upsilon|} \sum_{(a,b) \in \Upsilon}
p_\infty^{-1}\left[\sum_{x\in \rcal} \pi_x 
\left(\pi_x + (1 - \pi_x)e^{-\Lambda_i\dist(a,b)}\right)
 - q_\infty
\right]\\
&=& \frac{1}{|\Upsilon|} \sum_{(a,b) \in \Upsilon} e^{-\Lambda_i\dist(a,b)},
\end{eqnarray*}
which is {\em strictly decreasing in $\Lambda_i$}.
We need two properties for (\ref{eq:ucal}) to make a good site clustering
statistic: separation and concentration.

For {\em separation}, that is, for the statistic
above to distinguish different scaling factors as much
as possible, we require the following condition:
\begin{enumerate}

\item[S1] Each pair in $\Upsilon$ is composed of two ``sufficiently close'' leaves, that is,
there is $\alpha < +\infty$ such that
$\Upsilon \subseteq \Upsilon_\alpha$.

\end{enumerate}
Indeed, if two leaves are far away, their joint distribution is close to independent
and scaling has little effect on their agreement. A much better separation
is obtained from close leaves.

To guarantee {\em concentration} of a statistic of the type (\ref{eq:ucal}),
we require the following three conditions on $\Upsilon$:
\begin{enumerate}

\item[C1] The set
$\Upsilon$ is ``large enough'' and 
each pair makes a ``small contribution'' to the sum.
This will be satisfied 
if we show
that we can take
$|\Upsilon| = \Theta(n)$,
as (\ref{eq:ucal}) is a sum
of $\{0,1\}$-variables.

\item[C2] Agreement for different pairs in $\Upsilon$ 
is ``sufficiently uncorrelated,'' e.g., independent.

\end{enumerate}
These conditions will allow us to apply standard 
large deviations arguments.

\begin{example}[Full sum]
As a first guess, one may expect that taking
$\Upsilon$ to be {\em all pairs of leaves}
may give a good site clustering statistic. 
However, in general 
this is not the case as we show in the
following example. Consider the two-state
case on a complete
binary tree with identical edge lengths $\mu$
and $\Lambda = 1$.
For mathematical convenience, assume that the
states are $\rcal = \{+1,-1\}$ and let
$$
\gamma = 2 e^{-2\mu}.
$$
Then, up to a multiplicative factor and
additive constant, the clustering statistic
is simply
$$
\ucal^{(h)} = \sum_{(a,b)\in \nsqneq} \sigma_a \sigma_b,
$$ 
for a tree with $h$ levels.
Using a calculation of~\cite[Section 5]{EvKePeSc:00},
one has
\begin{equation}\label{eq:ekps}
\E[\sigma_a \sigma_b] = e^{-\dist(a,b)}.
\end{equation}
Dividing the expectation into terms
over the first subtree of the root,
terms over the second subtree of the root,
and terms between the two subtrees, we have
\begin{equation*}
\E\left[\ucal^{(h)}\right]
= 2\E\left[\ucal^{(h-1)}\right] + (2^{h-1})^2 e^{-2h\mu}.
\end{equation*}
Solving for the recursion gives
$$
\E\left[\ucal^{(h)}\right] 
= \gamma 2^{h-2} \frac{\gamma^h - 1}{\gamma - 1}
= O\left(2^{2h} \gamma^{2h}\right),
$$
as $h \to \infty$.
On the other hand,
the expectation of the square
$\E[(\ucal^{(h)})^2]$
is a sum of terms of the form
$\E[\sigma_{z_1}\sigma_{z_2}\sigma_{z_3}\sigma_{z_4}]$
where some of the $z$'s may be repeated.
All such terms are non-negative because
of~\eqref{eq:ekps} and the fact that
terms where all $z$'s are different factor
into a product by Proposition~\ref{claim:independence}
below. Hence
\begin{eqnarray*}
\var\left[\ucal^{(h)}\right]
&\geq& \sum_{(a,b)\in \nsqneq } \E\left[(\sigma_a \sigma_b)^2\right]
- \E\left[\ucal^{(h)}\right]^2\\
&=& \frac{2^h (2^h- 1)}{2} - \gamma^2 2^{2h-4} \left(\frac{\gamma^h - 1}{\gamma - 1}\right)^2\\
&=& \Omega(2^{2h}),
\end{eqnarray*}
if $\gamma < 1$. Hence, assuming that
$\gamma < 1$, we have
$$
\frac{\E\left[\ucal^{(h)}\right]^2}{\var\left[\ucal^{(h)}\right]}
\to 0,
$$
as $h \to \infty$. In other words, 
the sum over all pairs is too ``noisy'' to 
serve as a site clustering statistic in that case.
\end{example}

\subsection{Does it exist?}\label{section:exist}

We now show that there always exist statistics 
that satisfy the properties above
and we give explicit guarantee on their concentration.
Note that, in the current section, we only provide a proof
of {\em existence}. In particular,
in establishing existence, we use evolutionary distances
which are not available from the data.
Later, in Section~\ref{section:constructing}, we explain how to {\em construct} such a statistic 
from the data $\overline{\dcal}[T, \Lambda, r]$
(or, more precisely, the samples $(\sigma^{i}_{L})_{i=1}^k$)
{\em without knowledge of the tree topology,
evolutionary distances, or site scaling factors}.

We now explain how the conditions above 
can be achieved for an appropriate
choice of $\Upsilon$ on {\em any} tree topology.
Note, however, that $\Upsilon$ depends
on $T$. 

\paragraph{Independence}
We first show that the clustering statistic
(\ref{eq:ucal})
is a sum of independent variables {\em as 
long as the paths between different pairs
do not intersect.} This will allow us to satisfy C2.
\begin{proposition}[Independence]\label{claim:independence}
Assume that for all $(a,b), (a',b') \in \Upsilon$ with $(a,b) \neq (a',b')$
we have
\begin{equation*}
\path(a,b) \cap \path(a',b') = \emptyset,
\end{equation*}
where $\path(a,b)$ is the set of edges on the path between $a$ and $b$.
Then the random variables
$\left\{\ind\{\sigma_a = \sigma_b\}\right\}_{(a,b)\in \Upsilon}$,
are mutually independent.
\end{proposition}
\begin{proof}
Denote $\Upsilon = \{(a_1,b_1),\ldots,(a_\upsilon,b_\upsilon)\}$
with $\upsilon = |\Upsilon|$.
Let $\vcal$ be the set of nodes on the path between
$a_1$ and $b_1$.
Removing the edges
in $\path(a_1,b_1)$ creates a forest where the
$\vcal$-nodes can be taken as roots.
Note that, by symmetry and the Markov property,
conditioned on $\vcal$,
the distribution of the random variables
\begin{equation}\label{eq:independence1}
\left\{\ind\{\sigma_a = \sigma_b\}\right\}_{(a,b)\in \Upsilon
-\{(a_1,b_1)\}},
\end{equation}
does not depend on the states
of the $\vcal$-nodes. 
In particular, $\ind\{\sigma_{a_1} = \sigma_{b_1}\}$
is independent of~\eqref{eq:independence1}.
Proceeding by induction gives the result.
\end{proof}

\paragraph{Size}
To satisfy S1, we restrict ourselves to ``close pairs.''
We first show that the size of $\Upsilon_\alpha$ 
grows linearly {\em as long as $\alpha \geq 4g$},
allowing us to also satisfy C1.
A similar result is proved in~\cite{SteelSzekely:06}.
\begin{proposition}[Size of $\Upsilon_\alpha$]\label{claim:size}
Let $\alpha \geq 4g$. 
Then 
$$
\left|\Upsilon_\alpha\right| \geq \frac{n}{4}.
$$
\end{proposition}
\begin{proof}
Let
\begin{displaymath}
\Gamma = \{a \in [n]\ :\ \dist(a,b) > \alpha,\ \forall b\in [n]-\{a\}\},
\end{displaymath}
that is, $\Gamma$ is the set of leaves with no other leaf at evolutionary distance $\alpha$.
We will bound the size of $\Gamma$.
For $a \in \Gamma$, let
\begin{displaymath}
\bcal(a) = \left\{v \in V\ :\ \dist(a,v) \leq \frac{\alpha}{2}\right\}.
\end{displaymath}
Note that for all $a,b \in \Gamma$ with $a\neq b$ we have
$\bcal(a) \cap \bcal(b) = \emptyset$
by the triangle inequality. Moreover, it holds that for all $a \in \Gamma$
\begin{displaymath}
\left|\bcal(a)\right| \geq 2^{\left\lfloor\frac{\alpha}{2g}\right\rfloor},
\end{displaymath}
since $T$ is binary and there is no leaf other than $a$ in $\bcal(a)$.
Hence, we must have
\begin{displaymath}
|\Gamma| \leq \frac{2n - 2}{2^{\left\lfloor\frac{\alpha}{2g}\right\rfloor}}
\leq \left(\frac{1}{2^{\left\lfloor\frac{\alpha}{2g}\right\rfloor - 1}}\right) n,
\end{displaymath}
as there are $2n - 2$ nodes in $T$.

Now, for all $a \notin \Gamma$ assign an arbitrary leaf at evolutionary distance at most $\alpha$.
Then
\begin{eqnarray*}
\left|\Upsilon_\alpha\right|
&\geq& \frac{1}{2}(n - |\Gamma|)\\
&\geq& \frac{1}{2}\left(1 - \frac{1}{2^{\left\lfloor\frac{\alpha}{2g}\right\rfloor - 1}}\right) n,
\end{eqnarray*}
where we divided by $2$ to avoid over-counting.
The result follows from the assumption $\alpha \geq 4g$.
\end{proof}

\paragraph{Sparsification}
Note that $\Upsilon_{4g}$ satisfies C1 but does not satisfy C2 as the pairs
may be intersecting
(see Proposition~\ref{claim:independence}). We now show how to satisfy both C1 and C2 by ``sparsifying''
$\Upsilon_{4g}$. In stating this procedure, we allow some flexibility (that is, arbitrary choices) which will be useful
in analyzing the actual implementation in the next section.
Let $4g < m < M$ be a constant to be determined later and 
assume $\Upsilon'$ is any set
satisfying
\begin{equation*}
\Upsilon_{4g} \subseteq \Upsilon' \subseteq \Upsilon_m.
\end{equation*}
We know from
Proposition~\ref{claim:size} that $\Upsilon'$ has linear size, that is,
$|\Upsilon'| \geq n/4$.
We construct a linear-sized subset 
$\Upsilon$ of $\Upsilon'$ satisfying the 
non-intersection condition of Proposition~\ref{claim:independence}
as follows. Let $S := \Upsilon'$ and $\Upsilon'' := \emptyset$.
\begin{itemize}
\item Take any pair $(a^*,b^*)$ in $S$ and add it to $\Upsilon''$.

\item Let $S_0$ be any subset of $S$ such that
$S_0$ contains all pairs with at least one node within evolutionary distance $m$
of either $a^*$ or $b^*$ 
and contains no pair with both nodes 
beyond evolutionary distance
$M$ from both $a^*$ and $b^*$. Remove $S_0$ from $S$.

\item Repeat until $S$ is empty.

\item Return $\Upsilon := \Upsilon''$.

\end{itemize}
We claim that $\Upsilon$ is linear in size and that no two pairs in $\Upsilon$ intersect.
\begin{proposition}[Properties of $\Upsilon$]\label{claim:upsilon2}
Let $\Upsilon$ be any set built by the procedure above.
Then,
\begin{enumerate}
\item
For all $(a,b), (a',b') \in \Upsilon$ with $(a,b) \neq (a',b')$
we have
$\path(a,b) \cap \path(a',b') = \emptyset$.

\item
There is $\gamma_s = \gamma_s(M,f) > 0$ such that
$\left|\Upsilon\right| \geq \gamma_s n$,
where $\gamma_s$ does not depend on $T$,
but only on $M, f$.

\item For all $(a,b) \in \Upsilon$, we have
$$
2f \leq \dist(a,b) \leq M.
$$

\end{enumerate}
\end{proposition}
\begin{proof}
We first prove the non-intersecting condition.
All pairs of leaves in $\Upsilon$ are at 
evolutionary distance at most $m$.
Moreover, for any $(a,b)\neq (a',b')$ in $\Upsilon$, we have by construction
\begin{displaymath}
\min\{\dist(u,v):u\in\{a,b\},v\in\{a',b'\}\} \geq m.
\end{displaymath}
Hence, the path between $a$ and $b$ and the path between $a'$ and $b'$ cannot intersect: we have
$$
\dist(a,b) + \dist(a',b') - \dist(a,a') - \dist(b,b') 
\leq 0,
$$
which, using $\weight_e > 0$ for all $e$
and the four-point test 
(see e.g.~\cite{SempleSteel:03}),
excludes the topology $aa'|bb'$ (that is, the 
four-leaf topology
where $\{a,a'\}$ is one side of the internal edge
and $\{b,b'\}$ is on the other); and similarly 
for the topology $ab'|a'b$.

We now bound the size of $\Upsilon$.
Let $(a,b)$ be a pair of leaves at evolutionary
distance at most $m$. There are at most
$2\cdot 2^{\left\lfloor\frac{M}{f}\right\rfloor-1} = 2^{\left\lfloor\frac{M}{f}\right\rfloor}$ leaves at evolutionary distance 
at most $M$ from either $a$ or $b$.
Therefore, at each iteration of the sparsification algorithm, 
the number of
elements of $S$ removed is at most $2^{\left\lfloor\frac{M}{f}\right\rfloor + \left\lfloor\frac{m}{f}\right\rfloor - 1} \leq 2^{2\left\lfloor\frac{M}{f}\right\rfloor}$,
as each leaf removed is involved in at most $2^{\left\lfloor\frac{m}{f}\right\rfloor-1}$ pairs at evolutionary distance $m$.
Since by Proposition~\ref{claim:size}, the size of $\Upsilon'$ is at least $n/4$,
the number of elements in $\Upsilon$ at the end
of the sparsification algorithm
is at least 
$$
|\Upsilon|
\geq \frac{n}{2^{2\left\lfloor\frac{M}{f}\right\rfloor + 2}}.
$$

Finally, by our assumption on the phylogeny,
two distinct leaves are always at evolutionary distance at
least $2f$. For the upper bound, use $m < M$. 
\end{proof}
Define
$$
\gamma_s = \gamma_s(M,f) = \frac{1}{2^{2\left\lfloor\frac{M}{f}\right\rfloor + 2}}.
$$
\begin{definition}[Sparse pairs]
We say that a set $\Upsilon \subseteq \nsqneq$ 
is {\em $\gamma_s$-sparse} if it satisfies the three
properties in the statement of 
Proposition~\ref{claim:upsilon2}.
\end{definition}

\subsection{Properties of the site clustering statistic}

In (\ref{eq:ucal}) fix an $\gamma_s$-sparse $\Upsilon$.
We now show that conditions C1 and C2 lead to concentration.
Let 
$$
U_\Upsilon = \expec[\ucal_i]
= \frac{1}{|\Upsilon|} \sum_{(a,b) \in \Upsilon} \E\left[e^{-\Lambda_i\dist(a,b)}\right]
= \frac{1}{|\Upsilon|} \sum_{(a,b) \in \Upsilon} \mgf(\dist(a,b)).
$$
Moreover,
for $\lambda \geq 0$, define
$$
U_\Upsilon(\lambda)
= \frac{1}{|\Upsilon|} \sum_{(a,b) \in \Upsilon} e^{-\lambda\dist(a,b)}
$$
and note that
$$
U_\Upsilon(\Lambda_i)
= \E[\ucal_i\,|\,\Lambda_i],
$$
and
$$
U_\Upsilon = \E\left[U_\Upsilon(\Lambda_i)\right],
$$
for $i = 1,\ldots, k$.
\begin{proposition}[Concentration of $\ucal_i$]\label{claim:concentration}
For all $\zeta > 0$, there is $c > 0$ depending
on $M, f$ such that
\begin{equation*}
\prob\left[|\ucal_i - U_\Upsilon(\Lambda_i)| \geq \zeta \,|\,\Lambda_i\right]
\leq 2 \exp(-c \zeta^2 n),
\end{equation*}
almost surely, for all $i = 1,\ldots,k$.
(We will eventually use $\zeta = o(1)$.)
\end{proposition}
\begin{proof}
Recall the following standard concentration inequality (see e.g.~\cite{MotwaniRaghavan:95}):
\begin{lemma}[Azuma-Hoeffding Inequality]\label{lemma:azuma}
Suppose $X=(X_1,\ldots,X_m)$ are independent random variables taking values in a set
$S$, and $f:S^m \to \real$ is any $t$-Lipschitz function:
$|f(\bfx) - f(\bfy)|\leq t$ whenever $\bfx$ and $\bfy$ differ at just one coordinate. Then,
$\forall \zeta > 0$,
\begin{equation*}
\prob\left[|f(X) - \expec[f(X)]| \geq \zeta \right]
\leq 2\exp\left(-\frac{\zeta^2 }{2 t^2 m}\right).
\end{equation*}
\end{lemma}
From Propositions~\ref{claim:independence},~\ref{claim:size},
and~\ref{claim:upsilon2}, the random variable
$\ucal_i$ is a (normalized) sum of $\Omega(n)$ independent 
bounded variables.
By Lemma~\ref{lemma:azuma}, conditioning on
$\Lambda_i$,
we have $|U_\Upsilon(\Lambda_i) - \ucal_i| \leq \zeta$ except with probability $\exp(-\Omega(\zeta^2 n))$, where we used that
$m = \Omega(n)$ and $t = O(1/n)$.
\end{proof}

Moreover, we show separation. 
\begin{proposition}[Separation of $\ucal_i$]\label{claim:separation}
If $\lambda - \lambda' \geq \beta$, where $\beta \geq 0$,
then
$$
U_\Upsilon(\lambda') - U_\Upsilon(\lambda) 
\geq e^{-\lambda M} \left(
e^{2f \beta} - 1
\right).
$$
\end{proposition}
\begin{proof}
We have
\begin{eqnarray*}
U_\Upsilon(\lambda')
- U_\Upsilon(\lambda)
&=& \frac{1}{|\Upsilon|} \sum_{(a,b) \in \Upsilon} 
\left[e^{-\lambda'\dist(a,b)} - e^{-\lambda\dist(a,b)}\right]\\
&\geq& \frac{1}{|\Upsilon|} \sum_{(a,b) \in \Upsilon} 
\left[e^{-(\lambda - \beta)\dist(a,b)} - e^{-\lambda\dist(a,b)}\right]\\
&\geq& \frac{1}{|\Upsilon|} \sum_{(a,b) \in \Upsilon} 
e^{-\lambda\dist(a,b)}\left[e^{\beta\dist(a,b)} - 1\right]\\
&\geq& \frac{1}{|\Upsilon|} \sum_{(a,b) \in \Upsilon} 
e^{-\lambda M}\left[e^{\beta \cdot 2f} - 1\right],
\end{eqnarray*}
since $2f \leq \dist(a,b) \leq M$ for all
$(a,b) \in \Upsilon$ by assumption.
\end{proof}
\begin{remark}
The last bound may seem problematic because
under our assumptions 
the scaling factor is allowed to have an unbounded support.
In that case the RHS could be arbitrarily close to $0$. 
But we will show below that we can safely ignore
large values of $\lambda$. 
\end{remark}

\section{Constructing 
the site clustering statistic 
from data}
\label{section:constructing}

Note that in the previous
section we only established the {\em existence}
of an appropriate site clustering statistic.
We now show how such a statistic 
can be built from data 
{\em without knowledge of the tree topology or site scaling factors}. 

\paragraph{Notation} 
We use the same notation as in the previous section.
Further,
we let
\begin{displaymath}
\hat q(a,b) = \frac{1}{k}\sum_{i=1}^k
p_\infty^{-1}\left[\ind\{\sigma^i_a = \sigma^i_b\}
- q_\infty\right].
\end{displaymath}
Also let
\begin{eqnarray*}
q(a,b)
&=& \E[\hat q(a,b)]\\
&=& \E\left[p_\infty^{-1}\left[\ind\{\sigma^1_a = \sigma^1_b\}
- q_\infty\right]\right]\\
&=& \E\left[\E\left[p_\infty^{-1}\left[\ind\{\sigma^1_a = \sigma^1_b\}
- q_\infty\right]\,|\,\Lambda_1\right]\right]\\
&=& \E\left[e^{-\Lambda_1 \dist(a,b)}\right]\\
&=& \mgf(\dist(a,b)).
\end{eqnarray*}
where we used our previous calculations.
We define some constants used in the algorithm
and its analysis
whose values will be justified below.
Let
\begin{displaymath}
\omega_m = e^{-5g}, \qquad \omega_m^+ = e^{-5.5 g},
\qquad \omega_m^- = e^{-4.5 g}.
\end{displaymath}
Also recall that $\mgf$ is strictly decreasing and let
$$
m = \mgf^{-1}(\omega_m).
$$
Note that by Jensen's inequality
and $\E[\Lambda_1] = 1$
$$
\mgf(5g) \geq e^{-\E[\Lambda_1] \cdot 5g} = e^{-5g},
$$
so that
$$
m \geq 5g > 4g,
$$
as assumed in the previous section.
Similarly, by assumption,
$$
m = \mgf^{-1}(5g) < \mgf^{-1}(e^{-6g}) \leq M,
$$
so that $m < M$. Finally let
$$
\eta
= \min\left\{
e^{-4g} - e^{-4.5g},
e^{-4.5g} - e^{-5g},
e^{-5g} - e^{-5.5g},
e^{-5.5g} - e^{-6g}
\right\}.
$$

\subsection{Site clustering algorithm}

We proceed in three steps, as in the idealized setting of Section~\ref{section:exist}.
However, unlike the idealized setting,
we do {\em not} assume the knowledge of 
evolutionary distances. 
Note in particular that it is not possible to estimate
$\dist(a,b)$ from the samples $(\sigma^{i}_{L})_{i=1}^k$
because the rate distribution is unknown.
Instead, we use
$\hat q(a,b)$ as a {\em rough} estimate
of how close $a$ and $b$ are in the tree.
This will suffice for our purposes, as we show
in the next subsection.
The algorithm is the following:
\begin{enumerate}

\item \emph{(Close Pairs)}
For all pairs of leaves $a,b \in [n]$, compute $\hat q(a,b)$ and set
\begin{displaymath}
\Upsilon' = \{(a,b) \in \nsqneq\ :\ \hat q(a,b) \geq \omega_m^-\}.
\end{displaymath}

\item \emph{(Sparsification)}
Let $S := \Upsilon'$ and $\Upsilon'' := \emptyset$.
\begin{itemize}
\item Take any pair $(a^*,b^*)$ in $S$ and add it to $\Upsilon''$.

\item Remove from $S$ all pairs $(a,b)$ such that
\begin{displaymath}
\max\{\hat q(c^*, c)\ :\ c^*\in \{a^*,b^*\}, c\in\{a,b\}\} \geq \omega_m^+.
\end{displaymath}

\item Repeat until $S$ is empty.

\end{itemize}

\item \emph{(Final Statistic)}
Return $\Upsilon := \Upsilon''$. 

\end{enumerate}

\subsection{Analysis of the clustering algorithm}

Let $\Upsilon$ be the set
returned by the previous algorithm.
We show that 
it is $\gamma_s$-sparse
with high probability.
\begin{proposition}[Clustering statistic]
\label{proposition:clusteringstatistic}
Under Assumption~\ref{assump:general},
for all $0 < \delta < 1$ there exists
a constant $0 < C < +\infty$ (depending
on $g$) such that the set of pairs $\Upsilon$
returned by the previous algorithm
is $\gamma_s$-sparse with probability $1 - \delta$
provided that the number of samples 
satisfies
$$
k \geq C \log n.
$$
Moreover, the algorithm runs in polynomial time.
\end{proposition}
\begin{proof}
We first prove that all $\hat q(a,b)$'s
are sufficiently accurate.
\begin{lemma}
For all $0 < \delta < 1$, there exists
a constant $0 < C < +\infty$ (depending
on $g$) such that 
$$
\left|\hat q(a,b) - q(a,b)\right|
\leq \eta,
$$
for all $(a,b) \in \nsqneq$
with probability $1 - \delta$
provided that the number of samples 
satisfies
$$
k \geq C \log n.
$$
\end{lemma}
\begin{proof}
For each $(a,b) \in \nsqneq$, $\hat q(a,b)$
is a sum of $k$ independent bounded variables.
By Lemma~\ref{lemma:azuma}, taking $\zeta = \eta$
we have 
$$
|\hat q(a,b) - q(a,b)| \leq \eta,
$$ 
except with probability $2\exp(-C' k)$
for some $C' > 0$ depending on $p_\infty$
and $\eta$.
Note that there are at most $n^2$ 
elements in $\nsqneq$ so that
the probability of failure is at most
$$
2 n^2 \exp(- C' \cdot C \log n) \leq \delta,
$$
for $C$ sufficiently large.
\end{proof}
We return to the proof of Proposition~\ref{proposition:clusteringstatistic}.
Assume that the conclusion of the previous lemma
holds.
Our goal is to prove that the site clustering algorithm
then follows the idealized sparsification procedure described
in Section~\ref{section:existence}.
\begin{enumerate}
\item
We first prove that the set
\begin{displaymath}
\Upsilon' = \{(a,b) \in \nsqneq\ :\ \hat q(a,b) \geq \omega_m^-\}.
\end{displaymath}
satisfies
\begin{equation*}
\Upsilon_{4g} \subseteq \Upsilon' \subseteq \Upsilon_m.
\end{equation*}
Let $(a,b)$ be such that $\dist(a,b) \leq 4g$. 
Then
$$
q(a,b) = \mgf(\dist(a,b)) \geq \mgf(4g) \geq e^{-4g},
$$
by monotonicity and Jensen's inequality.
Hence
$$
\hat q(a,b) 
\geq e^{-4g} - \eta 
\geq e^{-4g} - \left(e^{-4g} - e^{-4.5g}\right)
\geq e^{-4.5g}
= \omega_m^-,
$$
and
$\Upsilon_{4g} \subseteq \Upsilon'$.

Similarly, let $(a,b)$ be such that
$\dist(a,b) > m$. Then
$$
q(a,b) = \mgf(\dist(a,b)) < \mgf(m) = \omega_m,
$$
and
$$
\hat q(a,b) < \omega_m + \eta 
\leq e^{-5g} + \left(e^{-4.5g} - e^{-5g}\right)
= \omega_m^-,
$$
so that $\Upsilon' \subseteq \Upsilon_m$.

\item Let $\Upsilon''$ be the set
obtained during one of the iterations of Step 2 of the site clustering algorithm
and fix a pair $(a^*,b^*) \in \Upsilon''$.
We need to show that the set $S_0$ of pairs $(a,b)$ 
in $\Upsilon''$ such
that
\begin{equation}\label{eq:omegamplus}
\max\{\hat q(c^*, c)\ :\ c^*\in \{a^*,b^*\}, c\in\{a,b\}\} \geq \omega_m^+
\end{equation}
is such that it
contains all pairs with at least one node within evolutionary distance $m$
of either $a^*$ or $b^*$ 
and contains no pair with both nodes 
beyond evolutionary distance
$M$ from both $a^*$ and $b^*$.
In the first case, assume w.l.o.g.~that
$$
\dist(a,a^*) \leq m.
$$
Then, arguing as above,
$$
\hat q(a,a^*) 
\geq e^{-5g} - \left(e^{-5g} - e^{-5.5g}\right) 
= \omega_m^+,
$$
and (\ref{eq:omegamplus}) is satisfied.
In the second case,
for all $c \in \{a,b\}$
and $c^* \in \{a^*,b^*\}$
$$
\dist(c,c^*) > M,
$$
and
$$
\hat q(c,c^*) < e^{-6g} + \left(e^{-5.5g} - e^{-6g}\right)
= \omega_m^+,
$$
so that (\ref{eq:omegamplus}) is not satisfied.

\end{enumerate}
The two properties above guarantee that
the algorithm constructs a set $\Upsilon$
as in the idealized sparsification procedure
of Section~\ref{section:existence}.
In particular, $\Upsilon$ is $\gamma_s$-sparse
by Proposition~\ref{claim:upsilon2}.
\end{proof}

\section{Tree reconstruction}\label{section:full}

We now show how to use our site clustering 
statistic to build the tree itself. The algorithm is composed of
two steps: we first ``bin'' the sites according
to the value of the clustering statistic; we then use the sites in
one of those bins and apply a standard distance-based
reconstruction method. By taking the bins 
sufficiently small, we show that the content 
of the bins is made of
sites with almost identical scaling factor---thus
essentially reducing the situation to the unmixed case.

Throughout this section, we assume that there
is a $\gamma_k > 0$ such that $k = n^{\gamma_k}$. We also
assume that $\Upsilon$ is $\gamma_s$-sparse,
that $\ucal$ stands for a copy of the corresponding
clustering statistic under scaling factor $\Lambda$,
and that Assumption~\ref{assump:general}
is satisfied.

\subsection{Site binning}

\paragraph{Ignoring small and large scaling factors}
We first show that, under Assumption~\ref{assump:general}, the scaling factor has non-negligible mass between two bounded values.
\begin{proposition}[Bounding the scaling factor]\label{prop:extreme1}
We have$$
\P\left[\minlambda \leq \Lambda \leq \maxlambda\right]
\geq \chi,
$$
where
$$
\minlambda = \frac{g}{M},
$$
$$
\maxlambda = \frac{2}{1 - e^{-5g}},
$$
and
$$
\chi = \frac{1 - e^{-5g}}{2}.
$$
\end{proposition}
\begin{proof}
From our convention
that $\E[\Lambda] = 1$,
Markov's inequality implies that
$$
\P[\Lambda \geq \maxlambda] 
\leq \frac{1}{\maxlambda}
= \frac{1-e^{-5g}}{2}.
$$

For the other direction, we reproduce
the argument in Remark~\ref{rem:zero}.
Recall that we assume that
$$
\mgf^{-1}\left(e^{-6g}\right) \leq \maxd. 
$$
Then the probability mass $\delta$ 
below $\eps$ (for $\eps < 6g/M$) must be such that
$$
\delta \leq e^{-(6g - \eps M)},
$$
since
$$
e^{-6g} \geq \mgf(M) \geq \delta e^{-\eps M}. 
$$
Take $\eps = g/M$ so that $\delta \leq e^{-5g}$.

Then we have
$$
\P\left[\minlambda \leq \Lambda \leq \maxlambda\right]
\geq 1 - \left(e^{-5g}\right) - \left(\frac{1-e^{-5g}}{2}\right)
= \chi,
$$
as desired.
\end{proof}
Translating the previous proposition into 
a statement about $\ucal$-values, 
we obtain the following.
\begin{proposition}[Bounding $\ucal$-values]
\label{prop:extreme2}
Letting $\chi$ be as above, we have
$$
\P\left[\minuu \leq 
U_\Upsilon(\Lambda)
\leq \maxuu\right] \geq \chi,
$$
where
$$
\minuu = e^{- M \maxlambda},
$$
and
$$
\maxuu = e^{-2 f \minlambda}.
$$
\end{proposition}
\begin{proof}
Recall that
$$
U_\Upsilon(\Lambda)
= \E[\ucal\,|\,\Lambda]
= \frac{1}{|\Upsilon|} \sum_{(a,b) \in \Upsilon} e^{-\Lambda\dist(a,b)}.
$$
Since 
$$2f \leq \dist(a,b) \leq M$$ 
for all $(a,b) \in \Upsilon$,
we have
$$
e^{-M \Lambda}
\leq \frac{1}{|\Upsilon|} \sum_{(a,b) \in \Upsilon} e^{-\Lambda\dist(a,b)}
\leq e^{-2f \Lambda}.
$$
The result then follows
from Proposition~\ref{prop:extreme1}.
\end{proof}

\paragraph{Binning the sites}
We will now bin the sites whose clustering statistic
lie between $\minuu$ and $\maxuu$.
The previous proposition guarantees that
there is a positive fraction of such sites
in expectation.
Let 
$$\Delta_U = \frac{\gamma_U}{\log n},$$
be the size of the bins in $\ucal$-space,
where $\gamma_U > 0$ is a constant to be fixed later. 
To avoid taking integer parts, we assume
for simplicity that $\maxuu - \minuu$
is a multiple of $\Delta_U$. 
Let 
$$
N_U 
= \frac{\maxuu - \minuu + 2\Delta_U}{\Delta_U},
$$ 
be the number of bins.
(The extra $2 \Delta_U$ in the numerator accounts for estimation
error. See below.)
Note that $\maxuu - \minuu = \Theta(1)$
and therefore $N_U = \Theta(\log n)$.
We proceed as follows:
\begin{itemize}
\item \textbf{(Initialization)} 
For $j = 0, \ldots, N_U$,
\begin{itemize}
\item $\bin_j = \emptyset$.
\end{itemize}
\item \textbf{(Main Loop)} For $i = 1, \ldots, k$,
\begin{itemize}
\item \textit{(Out-of-bounds)} 
If $\ucal_i \notin [\minuu - \Delta_U, \maxuu + \Delta_U)$
then $\bin_0 := \bin_0 \cup \{i\}$.
\item \textit{(Binning)} Else if 
$$
\ucal_i \in [
\minuu-\Delta_U + (j-1) \Delta_U,
\minuu-\Delta_U + j \Delta_U
)
$$
then $\bin_j := \bin_j \cup \{i\}$
and $\bin_{>0} := \bin_{>0} \cup \{i\}$.
\end{itemize}
\end{itemize}
Restating Proposition~\ref{claim:concentration},
we have:
\begin{proposition}[Concentration of $\ucal$-values]\label{lemma:ucalaccurate}
We have
\begin{equation}\label{eq:ucalaccurate}
|\ucal_i - U_\Upsilon(\Lambda_i)| \leq \Delta_U,
\qquad \forall i \in \{1,\ldots, k\},
\end{equation}
except with probability $\exp\left(- \Omega(n/\log^2 n)\right)$. 
\end{proposition}
\begin{proof}
Taking $\zeta = \Delta_U$ in Proposition~\ref{claim:concentration}, 
(\ref{eq:ucalaccurate}) holds except
with probability 
$$
2 n^{\gamma_k} \exp\left(- \Omega(n/\log^2 n)\right)
= \exp\left(- \Omega(n/\log^2 n)\right).
$$
\end{proof}
We first show that each bin contains
sites with roughly the same scaling factor.
We first need a bound on the scaling
factors in $\bin_{>0}$.
(Note that, because we needed that the bounds 
$\minuu$ and $\maxuu$ be independent 
of $\Upsilon$ (which itself depends on unknown
evolutionary distances), 
Proposition~\ref{prop:extreme1} does
not apply directly here.)
\begin{proposition}[Bounds on selected scaling factors]
\label{prop:boundsscaling}
Assume (\ref{eq:ucalaccurate}) holds.
There is $\gamma_U > 0$ small enough so 
that for all $i \in \bin_{>0}$
$$
\frac{fg}{M^2} \leq
\Lambda_i \leq \frac{2 M }{f(1 - e^{-5g})}.
$$
\end{proposition}
\begin{proof}
Since $U_\Upsilon(\Lambda_i) \leq \maxuu + 2\Delta_U$
by (\ref{eq:ucalaccurate}),
we have
\begin{eqnarray*}
\maxuu + 2\Delta_U
&\geq& 
\frac{1}{|\Upsilon|} \sum_{(a,b) \in \Upsilon} e^{-\Lambda_i\dist(a,b)}\\
&\geq& e^{-M \Lambda_i}.
\end{eqnarray*}
Choose $\gamma_U > 0$ small enough,
so that
$$
\maxuu + 2 \Delta_U \leq e^{-f \minlambda}.
$$
Taking logarithms,
$$
\Lambda_i \geq \frac{fg}{M^2}.
$$

Similarly, 
since $U_\Upsilon(\Lambda_i) \geq \minuu - 2\Delta_U$
by (\ref{eq:ucalaccurate}),
we have
\begin{eqnarray*}
\minuu - 2\Delta_U
&\leq& 
\frac{1}{|\Upsilon|} \sum_{(a,b) \in \Upsilon} e^{-\Lambda_i\dist(a,b)}\\
&\leq& e^{- 2 f \Lambda_i}.
\end{eqnarray*}
Choose $\gamma_U > 0$ small enough,
so that
$$
\minuu - 2 \Delta_U \geq e^{-2 M \maxlambda}.
$$
Taking logarithms,
$$
\Lambda_i \leq \frac{2 M }{f(1 - e^{-5g})}.
$$
\end{proof}
For $j\in\{1,\ldots,N_U\}$, 
let
$$
U_j = \minuu - \Delta_U + (j-1+1/2)\Delta_U,
$$
be the midpoint of the $j$-th bin.
Using the fact that $U_\Upsilon(\lambda)$
is strictly decreasing in $\lambda \in \real_+$,
we define $\lambda_j$ as the unique solution to
$$
U_\Upsilon(\lambda_j) 
= U_j,
$$ 
for $j \in \{1,\ldots,N_U\}$.
\begin{proposition}[Bin variation]\label{lemma:variation}
Assume (\ref{eq:ucalaccurate}) holds.
For any $\gamma_\Lambda > 0$,
one can pick $\gamma_U > 0$ small enough
(depending on $M, f, g$)
such that for any $j \in \{1,\ldots, N_U\}$
and $i \in \bin_j$,
$$
\left|\Lambda_i 
- \lambda_j\right|
\leq \frac{\gamma_\Lambda}{\log n}.
$$
\end{proposition}
\begin{proof}
This follows from Proposition~\ref{claim:separation}.
Assume that
$U_j \geq U_\Upsilon(\Lambda_i)$. 
(The other case is similar.)
Using the upper bound in Proposition~\ref{prop:boundsscaling} and (\ref{eq:ucalaccurate}),
we get
$$
\frac{3}{2} \Delta_U 
\geq
U_j - U_\Upsilon(\Lambda_i) 
\geq e^{-\Lambda_i M} \left(
e^{2f \beta} - 1
\right)
\geq 
\left(
e^{2f \beta} - 1
\right)
\exp\left(-\frac{2 M^2 }{f(1 - e^{-5g})}\right).
$$
Hence, 
$$
\beta  
\leq \frac{1}{2f} \log \left(1 + 
\frac{3\gamma_U\exp\left(\frac{2 M^2 }{f(1 - e^{-5g})}\right) }{2\log n}\right).
$$
\end{proof}
Next, we argue that at least one bin
contains a non-negligible fraction of sites.
We say that a bin $\bin_j$ is {\em abundant}
if 
$$
\left|\bin_{j}\right|
\geq k \frac{\chi}{6 N_U}.
$$
\begin{proposition}[Abundant bin]\label{lemma:abundant}
We have
\begin{equation}\label{eq:abundant}
\exists j^{*} \in \{1,\ldots,N_U\}\mbox{ such that $\bin_{j^*}$ is abundant},
\end{equation}
except with probability $\exp(-\Omega(n^{\gamma_\delta}))$
for some $\gamma_\delta > 0$.
\end{proposition}
\begin{proof}
For the analysis, we introduce {\em fictitious bins} for 
the (unknown) expected $\ucal$-values.
That is, for $i = 1,\ldots, k$,
we let $i \in B_j$ if 
$$
U_\Upsilon(\Lambda_i) 
\in  
[
\minuu-\Delta_U + (j-1) \Delta_U,
\minuu-\Delta_U + j \Delta_U
),
$$
for some $j \in \{2,\ldots,N_U-1\}$, or $i \in B_0$
otherwise.

Then, there is $j^{**} \in \{2,\ldots,N_U-1\}$ such that
\begin{equation}\label{eq:fictitious}
\P\left[
U_\Upsilon(\Lambda) 
\in  
[
\minuu-\Delta_U + (j^{**}-1) \Delta_U,
\minuu-\Delta_U + j^{**} \Delta_U
)
\right]
\geq \frac{\chi}{N_U},
\end{equation}
that is, the probability that a site falls into
bin $B_{j^{**}}$ is at least $\chi/N_U$.
This
follows immediately from Proposition~\ref{prop:extreme2}
and the fact that the bins are disjoint and cover
the interval $[\minuu,\maxuu]$.

From Lemma~\ref{lemma:azuma}
and \eqref{eq:fictitious},
$$
\P\left[
\left|B_{j^{**}}\right|
\leq k\frac{\chi}{2 N_U}
\right]
\leq 2 \exp\left(
- \frac{\left(k\chi/2 N_U\right)^2}{2 k}
\right)
= \exp\left(-\Omega(n^{\gamma_k}/\log^2 n)\right).
$$
Therefore, 
if (\ref{eq:ucalaccurate}) holds, one of 
$\bin_{j^{**}-1}$, $\bin_{j^{**}}$, or
$\bin_{j^{**}+1}$ must contain
at least a third of the sites in $B_{j^{**}}$.
This occurs with probability at least
$1 - \exp(-\Omega(n^{\gamma_\delta}))$
for some $\gamma_\delta > 0$.
\end{proof}

\subsection{Estimating a distorted metric}

\paragraph{Estimating evolutionary distances}
We use an abundant bin to estimate evolutionary distances.
\begin{itemize}
\item \textbf{(Abundant bin)} Let $\bin^{*}$ be any bin with
at least $k \frac{\chi}{6 N_U}$ sites and
set 
$$k^* = \left|\bin^*\right|$$
and
$$
\lambda^* = \lambda_j,
$$
where $j$ is the index of $\bin^*$,
that is, $\lambda^*$ is the midpoint
of $\bin^*$.

\item \textbf{(Evolutionary distances)} For all $a \neq b \in L$,
compute
\begin{displaymath}
\hat q^*(a,b) = \frac{1}{k^*}
\sum_{i\in \bin^*}
p_\infty^{-1}\left[\ind\{\sigma^i_a = \sigma^i_b\}
- q_\infty\right].
\end{displaymath}
 
\end{itemize}
We prove that the $\hat q^*(a,b)$ is a 
good approximation of $e^{-\lambda^* \dist(a,b)}$.
\begin{proposition}[Accuracy]\label{lemma:accuracy}
Let $\gamma_\Lambda > 0, \gamma_q < \gamma_k/2$ be fixed constants.
There is a $\gamma_\delta > 0$ such that
the following hold
except with probability $\exp(- \Omega(n^{\gamma_\delta}))$:
\begin{enumerate}
\item There is at least one abundant bin.
Let $\bin^*$ be an arbitrary such bin.

\item And, for each $i \in \bin^*$
and for all $a\neq b \in L$,
\begin{equation}\label{eq:accuracy}
\left|
\hat q^*(a,b) - e^{-\lambda^* \dist(a,b)}
\right|
\leq 
\frac{1}{n^{\gamma_q}}
+ e^{-\lambda^* d(a,b)} 
\left(
e^{\frac{\gamma_\Lambda\dist(a,b)}{\log n}}
-1\right).
\end{equation}
\end{enumerate}
\end{proposition}
\begin{proof}
The result follows from
Propositions~\ref{lemma:ucalaccurate}, 
~\ref{prop:boundsscaling}, 
~\ref{lemma:variation} and \ref{lemma:abundant},
and the following lemma.
\begin{lemma}
Let $\gamma_\Lambda > 0$, $\gamma_q < \gamma_k/2$ be fixed constants.
Let 
$$
\tilde k \geq k\frac{\chi}{6 N_U},
$$
and
$$
\frac{fg}{M^2}
\leq 
\tilde{\lambda}, \tilde{\lambda}_{1}, \ldots,
\tilde{\lambda}_{\tilde k}
\leq \frac{2M}{f (1 - e^{-5g})}
$$ 
be such that,
for all $i$,
\begin{equation}\label{eq:tildelambda}
\left|\tilde{\lambda}_{i} 
- \tilde{\lambda}\right|
\leq \frac{\gamma_\Lambda}{\log n}.
\end{equation}
Let $\tilde{\sigma}^i_L \sim \overline{\dcal}[T, \tilde{\lambda}_i, r]$ independently for all $i$. Then, for all $a\neq b \in L$,
\begin{equation*}
\left|
\frac{1}{\tilde{k}}
\sum_{i = 1}^{\tilde{k}}
p_\infty^{-1}\left[\ind\{\tilde{\sigma}^i_a = \tilde{\sigma}^i_b\}
- q_\infty\right] - e^{-\tilde{\lambda} d(a,b)}
\right|
\leq
\frac{1}{n^{\gamma_q}}
+ e^{-\tilde{\lambda} d(a,b)} 
\left(
e^{\frac{\gamma_\Lambda\dist(a,b)}{\log n}}
-1\right).
\end{equation*}
except with probability $\exp\left(- \Omega(n^{\gamma_k-2\gamma_q}/\log n)\right)$. 
\end{lemma}
\begin{proof}
In Lemma~\ref{lemma:azuma}, take
$m = \tilde{k} = \Omega(n^{\gamma_k}/\log n)$, $t = \frac{1}{\tilde{k}}$, and
$\zeta = \frac{1}{n^{\gamma_q}}$. Then,
except with probability
$2 n^2 \exp\left(- \Omega(n^{\gamma_k-2\gamma_q}/\log n)\right)$,
\begin{equation*}
\left|
\frac{1}{\tilde{k}}
\sum_{i = 1}^{\tilde{k}}
p_\infty^{-1}\left[\ind\{\tilde{\sigma}^i_a = \tilde{\sigma}^i_b\}
- q_\infty\right] - 
\frac{1}{\tilde{k}}
\sum_{i = 1}^{\tilde{k}}
e^{- \tilde{\lambda}_i \dist(a,b)}\right|
\leq
\frac{1}{n^{\gamma_q}},
\end{equation*}
for all $a\neq b \in L$.
Moreover, by \eqref{eq:tildelambda},
\begin{eqnarray*}
\left|
e^{-\tilde{\lambda} d(a,b)} - 
\frac{1}{\tilde{k}}
\sum_{i = 1}^{\tilde{k}}
e^{- \tilde{\lambda}_i \dist(a,b)}\right|
&\leq& 
e^{-\tilde{\lambda} d(a,b)} 
\left|
1 - 
\frac{1}{\tilde{k}}
\sum_{i = 1}^{\tilde{k}}
e^{|\tilde{\lambda} - \tilde{\lambda}_i|\dist(a,b)}\right|
\\
&\leq& 
e^{-\tilde{\lambda} d(a,b)} 
\left|
1 - 
e^{\frac{\gamma_\Lambda\dist(a,b)}{\log n}}\right|.
\end{eqnarray*}
\end{proof}
\end{proof}

\paragraph{Tree construction}
To reconstruct the tree, 
we use a distance-based method
of~\cite{DaMoRo:09}. We require the following
definition.
\begin{definition}[Distorted metric~\cite{Mossel:07,KiZhZh:03}]\label{def:distorted metric}
Let $T = (V,E;L,\dist)$ be a phylogeny and let $\tau, \Psi > 0$. 
We say that $\distd : L\times L \to (0,+\infty]$ is a $(\tau, \Psi)$-\emph{distorted metric}
for $T$ or a $(\tau, \Psi)$-\emph{distortion} of $\dist$ if:
\begin{enumerate}
\item \itemname{Symmetry}
For all $u,v \in L$, $\distd$ is symmetric, that is,
\begin{equation*}
\distd(u,v) = \distd(v,u);
\end{equation*}
\item \itemname{Distortion} $\distd$ is accurate on
``short'' distances, that is, for all $u,v \in L$, if either 
$\dist(u,v) < \Psi + \tau$ or $\distd(u,v) < \Psi + \tau$ then
\begin{equation*}
\left|\dist(u,v) - \distd(u,v)\right| < \tau.
\end{equation*}
\end{enumerate}
\end{definition}
An immediate consequence of~\cite[Theorem 1]{DaMoRo:09}
is the following.
\begin{theorem}[See~\cite{DaMoRo:09}.]\label{thm:dmr}
Let $T = (V,E;L,\dist)$ be a phylogeny 
with $n$ leaves in
$\tcal_{f,g}$. Then topology of $T$ can be recovered
in polynomial time from a 
$(\tau,\Psi)$-distortion $\distd$ of $\dist$ as long as
$$
\tau \leq \frac{f}{5},
$$
and
$$
\Psi \geq 5 g \log n.
$$
(The constants above are not optimal
but will suffice for our purposes.)
\end{theorem}
See~\cite{DaMoRo:09} for the details
of the reconstruction algorithm.

We now show how to obtain a
$(f/5,5g\log n)$-distortion with high 
probability.
\begin{proposition}[Distortion estimation]\label{lemma:distortion}
There are $\gamma_U, \gamma_\Lambda, \gamma_q,
\gamma_k > 0$
so that, given that the conclusions
of Proposition~\ref{lemma:accuracy} hold,
then
$$
\distd(a,b) = -\ln \left( \hat{q}^*(a,b)_+ \right),
\qquad (a,b) \in L\times L,
$$
is a $(\lambda^*f/5,5\lambda^*g\log n)$-distortion of
$\lambda^*\dist$.
\end{proposition}
\begin{proof}
Define
$$
\short = \{(a,b) \in L \times L\,:\,
\dist(a,b) \leq 15 g \log n\},
$$
and
$$
\notshort = \{(a,b) \in L \times L\,:\,
\dist(a,b) > 12 g \log n\},
$$

Let $(a,b)\in\short$. Note that
$$
e^{-\lambda^* \dist(a,b)}
\geq \exp\left(-\left(
\frac{2M}{f(1-e^{-5g})}
\right)
15 g \log n\right)
\equiv \frac{1}{n^{\gamma_q'}},
$$
where the last equality is a definition.
Then, taking $\gamma_q$ (and hence $\gamma_k$) large enough
and $\gamma_\Lambda$ (and hence $\gamma_U$) small enough, from~\eqref{eq:accuracy} we have
\begin{equation*}
\left|
\distd(a,b)
- \lambda^* \dist(a,b)
\right|
\leq
\left(\frac{fg}{M^2}\right)\frac{f}{5} 
\leq
\frac{\lambda^*f}{5}.
\end{equation*}

Similarly,
let $(a,b)\in\notshort$. Note that
$$
e^{-\lambda^* \dist(a,b)}
< \exp\left(-\left(
\frac{fg}{M^2}
\right)
12 g \log n\right)
\equiv \frac{1}{n^{\gamma_q''}},
$$
where the last equality is a definition.
Then, taking $\gamma_q$ large enough
and $\gamma_\Lambda$ small enough,
from~\eqref{eq:accuracy} we have
\begin{equation*}
\distd(a,b)
\geq
6 \lambda^* g \log n
\geq 
5\lambda^*g \log n + \frac{\lambda^*f}{5}. 
\end{equation*}
\end{proof}

We finally state our main tree-construction result.
\begin{proposition}[Tree reconstruction]
\label{prop:treeconstruction}
Under Assumption~\ref{assump:general},
given a $\gamma_s$-sparse $\Upsilon$
there is a $\gamma_k>0$ large enough
so that the topology of the tree
can be reconstructed in polynomial
time using $k = n^{\gamma_k}$
samples, except with probability
$\exp(-\Omega(n^{\gamma_\delta}))$
for some $\gamma_\delta > 0$.
\end{proposition}
\begin{proof}
The result follows from
Theorem~\ref{thm:dmr}
and Proposition~\ref{lemma:distortion}.
\end{proof}
Combining Propositions~\ref{proposition:clusteringstatistic}
and~\ref{prop:treeconstruction},
we get Theorem~\ref{thm:2}.

\section{Concluding remarks}

Using techniques from the recent unpublished
manuscript~\cite{MosselRoch:11b},
our results can be extended to handle
the more general GTR model of molecular 
evolution which allows $Q$-matrices to be time-reversible. 
This generalization involves choosing pairs of leaves
that are not only connected by edge-disjoint 
paths, but also far enough from each other.
One can then use mixing arguments to derive the
independence properties required for concentration
of the site clustering statistic.
We leave out the details.

\clearpage

\bibliographystyle{alpha}
\bibliography{thesis}

\end{document}